\documentclass[makeidx]{amsart} 
\usepackage{amsmath}
\usepackage{amssymb}
\usepackage{epsfig}
\usepackage{hyperref}
\usepackage{color}
\usepackage{tikz}
\usetikzlibrary{arrows,patterns,decorations.pathmorphing}
\newtheorem{Proposition}{Proposition}
  \newtheorem{Remark}{Remark}
  
  \newtheorem{Lemma}[Proposition]{Lemma}
  
  \newtheorem{Theorem}{Theorem}
 
 \newtheorem{Definition}[Proposition]{Definition}

\def\blackslug{\hbox{\hskip 1pt \vrule width 4pt height 8pt depth 1.5pt
\hskip 1pt}}
\def\qed{\quad\blackslug\lower 8.5pt\null\par}

\def\r#1{\textcolor{red}{#1}}

\makeindex

\begin{document}

\author{S. Tanveer$^1$ and C. Tsikkou$^2$} 
\address{$^1$ Mathematics Department\\The Ohio State University\\Columbus, OH 43210 \\ 
$^2$ Mathematics Department\\West Virginia University \\Morgantown, WV 26506} 

\title{Analysis of $2+1$ diffusive-dispersive PDE arising in river braiding}

\maketitle
 
\begin{center}
\today
\end{center}
\begin{abstract}
We present local existence 
and uniqueness results for the following $2+1$ dispersive diffusive
equation due to Hall
\cite{Phall} arising in modeling of river braiding:
$$ u_{yyt} - \gamma u_{xxx} -\alpha u_{yyyy} - \beta u_{yy} + \left ( u^2 \right )_{xyy} = 0 $$
for $(x,y) \in [0, 2\pi] \times [0, \pi]$, $t> 0$, with boundary condition $u_{y}=0=u_{yyy}$ at $y=0$ and
$y=\pi$ and $2\pi$ periodicity in $x$, using a contraction mapping argument in a 
Bourgain-type
space $T_{s,b}$. We also show that the energy $\| u \|^2_{L^2} $ and cumulative dissipation
$\int_0^t \| u_y \|_{L^2}^2 dt$ are globally controlled in time.
\end{abstract}

\section{Introduction and the main result}

In the context of a weakly nonlinear study of instabilities of a straight river channel, Hall \cite{Phall}
introduced the following evolution equation for deposited sediment depth $u(x, y, t)$:
\begin{equation}
\label{1.1}
u_{yyt} - \gamma u_{xxx} - \alpha u_{yyyy} -\beta u_{yy} + [u^2]_{xyy} = 0   
\end{equation}
with parameters $\gamma, \alpha, \beta$ where $\alpha > 0$ and $\gamma \ne 0$. The domain
of interest is
$\mathcal{D}: \left \{ (x,y, t): y \in (0, 2 \pi), x \in \mathbb{R} ~ \ ,~ t > 0 \right \}$.
The boundary condition physically appropriate corresponds to
\begin{equation}
\label{1.2}
u_y = 0 = u_{yyy} ~{\rm on}~ y =0, \pi .  
\end{equation}
With initial conditions
\begin{equation}
\label{1.3}
u(x,y,0) = u_0 (x,y) \,
\end{equation} 
\eqref{1.1}-\eqref{1.3} constitutes the initial/boundary value problem of interest. 
We are not aware of any analysis of the Hall equations (\ref{1.1})-(\ref{1.3}).
In a mathematical context, this is an example of a nonlinear $2+1$ evolution equation that is dispersive
in one spatial direction ($x$), like the KdV or nonlinear Schroedinger equation, 
while being diffusive in the other direction ($y$). 
Further, if one were to express (\ref{1.1}) in an infinite system of $1+1$ equations using
a Fourier cosine series involving $\cos (ny)$  in $y$, as appropriate for \eqref{1.2},
the dispersive regularization in $x$ is not uniform. 
Therefore, the 
analysis of this initial value problem requires significant adaptation of
known methods (see for instance \cite{Bourgain}, \cite{Koenig16}, \cite{Tao}) for
dispersive PDEs.)
Indeed, the adaptation introduced here
should be of interest more generally to other $2+1$ diffusive-dispersive systems. 
Here we present results for the periodic case of the Hall equation:  
\begin{equation}
\label{1.3.1}
u (x+2 \pi, y, t) = u (x, y, t) \ ,
\end{equation}
though it 
will be clear that the analysis can be adapted with minor changes (for instance
sum over $m$ replaced by integration in the ensuing) for the non-periodic case $x \in \mathbb{R}$ as well. However,
finiteness in the $y$-direction is essential for the method presented here to work. We also
note that the analysis for the more general periodic case in $x$ and $y$ (different from $2 \pi$)
can be brought to the case studied here merely by rescaling
variables and parameters.
We note at once from $y$ integration of (\ref{1.1}) that   
for any regular solution
\begin{equation}
\label{1.3.1}
\frac{\partial^3}{\partial x^3} \int_0^{\pi} u (x, y, t) dy = 0 \ , {\rm implying~from~above}~ 
\frac{1}{\pi} \int_0^{\pi} u(x, y, t) dy = C_0 (t) . 
\end{equation}
It is to be noted that $C_0 (t)$ can be specified and is not determined by the equation itself; however,
the analysis only needs to be done for $C_0 (t)=0$ since 
if $C_0 (t) \ne 0$,
the change of variables $(x,y,t, u) \rightarrow \left ( x - 2 \int_0^t C_0(t') dt', y, t, u+C_0 (t) \right )$ 
transforms \eqref{1.1} back to itself with the new $u$ satisfying $\int_0^{\pi} u (x,y, t) dy=0$.
Therefore, we seek solution with representation\footnote{The insertion of factor $e^{\beta t}$ in (\ref{1.4}) 
makes the analysis of the ensuing integral equation
simpler.}
\begin{equation}
\label{1.4}
u(x, y, t) = e^{\beta t} 
\sum_{(m,n)\in \mathbb{Z}^2_0} u_{m,n} (t) \exp \left [ i m x + i n y \right ] \ ,
\end{equation}
where
\begin{equation}
\label{1.4.1}
\mathbb{Z}^2_0 = \left \{ (m,n) \in \mathbb{Z}^2: n \ne 0 \right \} \ ,
\end{equation}
with restriction 
\begin{equation}
\label{1.5}
u_{m,-n} = u_{m,n} \ , u_{-m,n} = u^*_{m,n}  
\end{equation}
that ensures that   
$u$ is real and contains only 
$\left \{ \cos (ny) \right \}_{n=1}^\infty$ 
terms that automatically
satisfy boundary conditions (\ref{1.2}) at
$y=0$ and $y=\pi$.
Applying Duhamel's principle, the initial boundary value problem (\ref{1.1})-(\ref{1.3}) 
can be formally reduced to the following integral equation
for 
$\mathbf{U} (t) = 
\left \{ u_{m,n} (t) \right \}_{m,n\in \mathbb{Z}^2_0} $ with $t > 0$:
\begin{equation}
\label{1.6}
u_{m,n} (t) = e^{-\left ( \alpha n^2 - i l_{m,n} \right ) t} u_{m,n} (0) 
+ \int_0^t e^{-\left (\alpha n^2 -i l_{m,n} \right ) (t-s)} e^{\beta s}  
A_{m,n} (s) ds \ , 
\end{equation}
where 
\begin{equation}
\label{1.6.1}
l_{m,n} = \frac{\gamma m^3}{n^2}  
\end{equation}
and
\begin{equation}
\label{1.7}
A_{m,n} = -i m \sum_{(m',n') \in \mathbb{Z}^2_{0,n}}
u_{m',n'} u_{m-m',n-n'} \ ,
\end{equation}
where for any $n \in \mathbb{Z} \setminus \{0 \}$,
\begin{equation}
\label{1.7.1}
\mathbb{Z}^2_{0,n} = \left \{ (m',n') \in \mathbb{Z}^2: n' \ne 0, n' \ne n \right \} \ .
\end{equation}
We also require that $u_{m,n} (0)$ to satisfy the symmetry condition (\ref{1.5}). 
\begin{Remark}
\label{rem1}
{\rm 
With the symmetry (\ref{1.5}) imposed initially,
with appropriate transformation of summation variables $(m',n')$, it is clear
that if ${\bf U} (t) = \left \{ u_{m,n} (t) \right \}_{(m,n)\in \mathbb{Z}^{2}_{0}}$ 
is one solution, so will be
$\left \{ u_{m,-n} (t) \right \}_{(m,n) \in \mathbb{Z}^{2}_0}$, or 
$\left \{ u^*_{-m,n} (t) \right \}_{(m,n) \in \mathbb{Z}^{2}_0}$
Therefore, if the solution is unique, as will be shown to be the case, 
the conditions (\ref{1.5}), once satisfied initially, 
remain time-invariant and therefore (\ref{1.5}) is satisfied automatically. It is of course
possible to generalize representation (\ref{1.6}) for complex periodic initial
data by relaxing (\ref{1.5}), though it is not of physical interest in the river context.}
\end{Remark}
It is useful to introduce abstract operator $e^{t \mathcal{L} }$ 
so that integral equation (\ref{1.5}) 
through basis representation (\ref{1.4}) can be interpreted as
\begin{equation}
\label{absDuhammel}
u (\cdot, \cdot, t) = e^{t \mathcal{L} } u_0 - \int_0^t e^{(t-\tau) \mathcal{L} } 
~~\left [u^2 \right ]_x (\cdot, \cdot, \tau)  d\tau    
\end{equation}
where in the basis representation (\ref{1.4})
\begin{equation}
\label{eqdefL}
\left [ e^{t \mathcal{L}} v \right ]_{m,n} = e^{\beta t} e^{(-\alpha n^2 + i l_{m,n} )t} v_{m,n}   
\end{equation}
\begin{Remark}
\label{rem2}
{\rm 
Equation (\ref{absDuhammel}) is an integral reformulation formulation of the original PDE
initial value problem (\ref{1.1}), (\ref{1.3}); it does not require
functions to be differentiable in any sense, except for once in $x$. Therefore,
it is appropriate to look for solutions to
(\ref{absDuhammel}) for which
$u_{m,n} (t)$ is integrable in time and
$ \left \{ (1+|m|+|n|)^{s} u_{m,n} \right \} \in l^2 (\mathbb{Z}^{2}_0)$, 
for $s > 2$.
Clearly if $s $ is large enough, it generates
a classical solution of (\ref{1.1}), (\ref{1.3}).
}
\end{Remark}

As usual for treatment of dispersive PDEs through Fourier
transform in time, we introduce a smooth cut-off function $\phi$ 
with support in $(-2 \delta, 2 \delta)$ and with $\phi =1$ in $[-\delta, \delta]$.
We replace the original system of equations (\ref{1.6}) by the following 
system that extends the
solution to $t \in \mathbb{R}$.
\begin{equation}
\begin{aligned}
\label{1.8}
\theta_{m,n} (t)  
&= \phi (t) e^{-\alpha |t| n^2 +i l_{m,n} t} u_{m,n} (0) \\
&+ \phi(t) \int_{\mathbb{R}} H(t-s) e^{-(\alpha n^2 -i l_{m,n} ) (t-s)} 
e^{\beta s} \phi^2 (s) \Lambda_{m,n} (s) ds \\
&-
\phi(t) e^{(i l_{m,n} t -\alpha n^2 |t|)}  
\int_{\mathbb{R}} H(-s) e^{(\alpha n^2  -i l_{m,n}) s}
e^{\beta s} \phi^2 (s) \Lambda_{m,n} (s) ds  \ ,  
\end{aligned}
\end{equation}
where $H$ is the Heaviside function ($H:=\chi_{[0, \infty)}$) and
\begin{equation}
\label{1.9}
\Lambda_{m,n} =-i m \sum_{(m',n') \in \mathbb{Z}^2_{0,n}}
\theta_{m',n'} \theta_{m-m',n-n'}
\end{equation}
It is clear that any
solution
$\boldsymbol{\theta} (t) := \left \{ \theta_{m,n} (t) \right \}_{(m,n) \in \mathbb{Z}^2_0}$ 
satisfying (\ref{1.8}) has compact support in $(-2 \delta, 2 \delta)$ and, since $\phi =1$ for
$t\in [0, \delta]$, satisfies exactly the same equation (\ref{1.6}) in that interval as does 
${\bf U} (t) = \left \{ u_{m,n} (t) \right \}_{(m,n)\in \mathbb{Z}^2_0}$. 
Further, for any solution
${\bf U} (t)$
to (\ref{1.6}), $\phi (t) {\bf U} (t)$ will be compactly supported in $(-2\delta, 2 \delta)$ and
for $t \in [0, \delta]$, where $\phi =1$, will 
satisfy (\ref{1.8}), though this is not
true outside this time interval. 
It is convenient to write the system in \eqref{1.8} symbolically as
\begin{equation}
\label{1.10}
\boldsymbol{\theta}  = \boldsymbol{\theta}^{(0)}  + \mathcal{M} [ \boldsymbol{\theta} ] 
=: \mathcal{N} [\boldsymbol{\theta} ]  \ ,    
\end{equation}
where the $(m,n)$-th component of (\ref{1.10}) at time $t \in \mathbb{R}$, for 
$(m,n) \in \mathbb{Z}^2_0$ 
is given by (\ref{1.8}), with
$\Lambda_{m,n} $ determined in terms of $\theta_{m,n}$ through (\ref{1.9}).

\begin{Definition}
{\rm
We define a Bourgain-type \cite{Bourgain} norm $\| . \|_{s,b}$ of
vector function $\boldsymbol{\theta} (t) = \left \{ \theta_{m,n} \right \}_{(m,n) 
\in {\mathbb{Z}^2_0}}$ such that
\begin{equation}
\label{1.11}
\| \boldsymbol{\theta} \|_{s,b}^2 = \sum_{(m,n) \in \mathbb{Z}^2_0 }
\int_{\mathbb{R}} W_{m,n} (\tau; b) \Big | \mathcal{F} \left [ 
\theta_{m,n} \right ] (\tau) \Big |^2 d\tau \ , 
\end{equation}
where $\mathcal{F} [\theta_{m,n} ] (\tau) = \int_{\mathbb{R}} e^{-i t \tau} \theta_{m,n} (t) dt$
is the Fourier transform in time and weight
\begin{equation}
\label{1.12}
W_{m,n} (\tau; b) = \rho_{m,n} \omega_{m,n} (\tau; b) \ , ~
\rho_{m,n} = \left ( |n|+|m| \right )^{2s} \ ,~~ 
\omega_{m,n} (\tau; b) = \left ( n^2 + 
|\tau-l_{m,n} | \right )^{2b} .
\end{equation}
The space of 
vector functions $\left \{ \boldsymbol{\theta} (t) \right \}$ for which 
$ \| \boldsymbol{\theta} \|_{s,b} < \infty$ is defined as $T_{s,b}$.
We also denote by $H^s_{0,b}$ 
the space of functions $v (x,y)$ in $\mathbb{T} [0, 2\pi ]^2$ with $\int_0^{2\pi}
v (x,y) dy = 0 $, represented by $v (x,y) = \sum_{(m,n) \in \mathbb{Z}^2_0} v_{m,n} e^{i m x + i n y}$, 
equipped with the norm
$$ \| v \|_{H^s_{0,b}}^2 = \sum_{(m,n) \in \mathbb{Z}^2_0} |n|^{4b-2} \rho_{m,n} |v_{m,n}|^2 $$
$H^s_0$ will denote the subspace of the usual Sobolev space $H^s$ satisfying zero average condition
$\int_0^{2\pi} v(x,y) dy=0$.
}
\end{Definition}
\begin{Remark}
{\rm 
With the representation (\ref{1.4}),
the $\| \cdot \|_{s,b}$ norm could also be thought of as a norm in the space
of 
functions $\left \{  v(x , y, t) \right \}$
rather than the space of corresponding sequences, 
as defined here.
}
\end{Remark}
\begin{Theorem} 
\label{Thm1}
Assume $s > \frac{5}{2}$ and $u_0 \in H^s_{0,b} (\mathbb{T}^2)$.  
Then for $b \in \left ( \frac{1}{2}, \frac{2}{3} \right )$, there exists $\delta$ sufficiently
small, that depends on $\| u_0 \|_{H^{s}_{0,b}}$, 
so that there exists unique solution $\boldsymbol{\theta} $  to the integral equation 
\eqref{1.10} in some small ball in $T_{s,b}$. If
initial data satisfies (\ref{1.5}), then 
this corresponds to 
the unique continuous solution 
$u (\cdot, \cdot, t) \in H^{s}_0 (\mathbb{T}^2 [0, 2 \pi])$ 
satisfying (\ref{1.1})-(\ref{1.3}) in the time interval of existence.
Further, the energy $ \| u (\cdot, \cdot, t) \|_{L^2}^2$ as well as cumulative
dissipation
$\int_0^t \| u_y (\cdot, \cdot, s) \|_{L^2}^2  ds$ are globally controlled in time.
\end{Theorem}
The proof will rely 
on a contraction argument in the space $T_{s,b}$ for small $\delta$ and 
is 
similar in spirit to the ones presented earlier in \cite{Tao}-\cite{KoenigPV}
for $1+1$ dispersive system like the KdV. What is
new here is the introduction of the space $T_{s,b}$ and a suitable
integral
reformulation (\ref{1.8}) for which dispersive regularization estimates 
through Fourier transform in time is good enough  
despite lack of uniform dispersion, and also overcomes ill-posedness in backwards time in
the original formulation (\ref{1.6}).
Also, crucial to the proof are the bounds established in \S \ref{prelm}, which rely partially on
previous estimates \cite{Tao}, \cite{KoenigPV}, 
or simple adaptation of them, which are presented in the Appendix for completeness.
The energy control and uniqueness relies on a more
traditional energy argument in classical Sobolev space.

\begin{Remark}
{\rm 
Clearly, since $b > 1/2$, $ H^s_0 \subset H^s_{0,b} \subset H^{s+(4b-2)}_0$, where
$H^s_0 $ is the subspace of the traditional Sobolev space $H^s$ with zero average in $y$. Indeed, 
since the condition for existence and uniqueness locally in time is
$ s > \frac{5}{2}$, with no restriction on
how small $b-1/2 > 0$ is, any initial data in $H^{s'}_0$ for $s' > \frac{5}{2}$ suffices.
}
\end{Remark}

\begin{Remark} 
\label{rem3}
{\rm
Global control on $\| u \|_{L^2}$ is however not enough to control $\| u (\cdot, \cdot, t )\|_{H_s^0}$ for $s > \frac{5}{2}$. Global control will be the
subject of future work.
}
\end{Remark}

\section{Key Lemmas and Proof of Theorem \ref{Thm1}}

\noindent{\bf Notation:} The symbol $ \lesssim$ is used in the following sense:  $|f| \lesssim  |g|$ is equivalent
to existence of some constant $C$ 
independent of $(m,n,m',n',\tau)$, but possibly dependent on $\alpha, \beta, \gamma, s, ~\text{and} ~b$ such that
$|f| \le C |g|$. The value of $C$ is not important.

\begin{Lemma}
\label{lemIC}
For $ \frac{1}{2} < b $,
\begin{equation}
\label{2.1}
\| \boldsymbol{\theta}^{(0)} \| 
\lesssim \| u_0 \|_{{H}^s_{0,b}}  
\end{equation}
\end{Lemma}
\begin{proof}
First, we note from computation that
$$\mathcal{F} \left [ \exp \left ( -\alpha n^2  |t| + i l_{m,n} t \right )\right ] (\tau) 
= \frac{2 \alpha n^2}{(\alpha^2 n^4 + \left ( \tau - l_{m,n} \right )^2}
$$
Therefore, 
$$
\int_{\mathbb{R}} \left ( n^2 + \Big |\tau-l_{m,n} \Big | \right )^{2b} 
\Big | \mathcal{F} \left [ e^{-\alpha n^2 |t| + i l_{m,n} t } \right ] (\tau)  \Big |^2
\lesssim \int_{\mathbb{R}} \frac{n^4 d\tau}{\left (n^2 + \Big | \tau - l_{m,n} \Big | \right )^{4-2b}} 
\lesssim |n|^{4b-2} 
$$
Using the above inequality and applying Lemma \ref{lem4.1} to  
$\boldsymbol{\theta}^{(0)}$ (see \eqref{1.8} and \eqref{1.10}), it follows that   
$$ \| \boldsymbol{\theta}^{(0)} \|_{s,b}^2 \lesssim   
\sum_{(m,n) \in \mathbb{Z}^2_{0}} |n|^{4b-2}
\left (|m| + |n| \right )^{2s}  |  u_{m,n} (0)  |^2  
\lesssim 
\| u_0 \|^2_{H^s_{0,b}} $$
\end{proof}
\begin{Proposition}
\label{prop0}
For $ b \in \left ( \frac{1}{2}, \frac{2}{3} \right ) $, $ b \le b' \le \min \{2 b-\frac{1}{2}, \frac{2}{3} \} $, and 
$s > \frac{5}{2} $,
\begin{equation}
\sup_{(m,n) \in \mathbb{Z}_{0}^2, \tau\in \mathbb{R}} 
\left [ \sum_{(m',n') \in \mathbb{Z}^{2}_{0,n} }
\int_{\mathbb{R}} \frac{m^2 W_{m,n} (\tau; b') 
d\tau_1}{|n^2 +i (\tau-l_{m,n})|^2 W_{m',n'} (\tau_1; b) W_{m-m',n-n'} (\tau-\tau_1; b)} 
\right ] \lesssim 1
\end{equation}
\end{Proposition}
The proof of Proposition \S \ref{prop0} is given at the end of 
\S \ref{prelm} after some bounds are established.
\begin{Lemma}
\label{lem4.5}
Define 
\begin{equation}
b_{m,n} (t) = \sum_{(m',n') \in \mathbb{Z}^2_{0,n} } 
\int_{\mathbb{R}} H(t-s) e^{(-\alpha n^2+i l_{m,n}) (t-s)}  {\hat \theta}_{m',n'} (s) {\hat \beta}_{m-m', n-n'} (s) ds 
\ ,
\end{equation}
where ${\hat \theta}_{m,n} (t) = \psi (t) \theta_{m,n} (t)$, ${\hat \beta}_{m,n} (t) = \psi (t) \beta_{m,n} (t)$,
$\psi (t) = e^{\beta t/2} \phi (t)$.
With the restriction on $b$, $b'$, $s$ in Proposition \ref{prop0},
\begin{equation}
\label{eqlem4.5}
\sum_{ (m,n) \in \mathbb{Z}^2_{0}} 
\int_{\mathbb{R}} W_{m,n} (\tau; b) m^2 \Big | \mathcal{F} [ \phi b_{m,n} ] (\tau) \Big |^2 d\tau  
\lesssim \delta^{2b'-2b} 
\| \boldsymbol{\theta} \|^2_{s,b}
\| \boldsymbol{\beta} \|^2_{s,b}
\end{equation}
\end{Lemma}
\begin{proof}
First, note that if we apply Lemma \ref{lem4.1} with $q = b_{m,n}$, it follows that the
left hand side of
\eqref{eqlem4.5} 
\begin{equation}
\label{eqlem4.5p}
\lesssim \delta^{2b'-2b} \sum_{ (m,n) \in \mathbb{Z}^2_{0}} 
\int_{\mathbb{R}} W_{m,n} (\tau;b') m^2 \Big | \mathcal{F} [ b_{m,n} ] (\tau) \Big |^2 d\tau  
\end{equation}
We note from convolution form of $b_{m,n}$ that  
\begin{equation}
\mathcal{F} [b_{m,n} ] (\tau) 
= \sum_{(m',n') \in \mathbb{Z}^2_{0}} \frac{1}{\alpha n^2+i (\tau-l_{m,n})} 
\int_{\mathbb{R}} \mathcal{F} [{\hat \theta}_{m',n'}] (\tau_1)  \mathcal{F} [{\hat \beta}_{m-m, n-n'}] (\tau-\tau_1) d\tau_1 
\end{equation}
So, applying Cauchy-Schwartz inequality it follows that the expression 
in \eqref{eqlem4.5p}
\begin{multline}
\lesssim 
\delta^{2b'-2b} \sup_{(m,n) \in \mathbb{Z}^2_0, \tau \in \mathbb{R}} 
\left \{ \sum_{(m',n')} \int_{\mathbb{R}} \frac{m^2 W_{m,n} (\tau; b') d\tau_1}{|n^2+i (\tau-l_{m,n})|^2 
W_{m',n'} (\tau_1; b) W_{m-m',n-n'} (\tau-\tau_1; b)} 
\right \} \\
\times \left \{ \sum_{(m,n)} \sum_{(m',n')} \int_{\mathbb{R}} 
\int_{\mathbb{R}} W_{m',n'} (\tau_1; b) \Big | \mathcal{F} [{\hat \theta}_{m',n'} ] (\tau_1) \Big |^2  \right .\\
\left.
\times W_{m-m',n-n'} (\tau-\tau_1; b) \Big | \mathcal{F} [{\hat \beta}_{m-m',n-n'}] (\tau-\tau_1) \Big |^2 
d\tau_1 d\tau\right \} 
\ ,
\end{multline}
from which the Lemma follows if we apply Proposition \ref{prop0} and Lemma \ref{lem4.1}, 
$q=\theta_{m',n'}$ or $q=\beta_{m-m',n-n'}$ and $\phi$ replaced by $\psi$, which is also compactly
supported smooth function in $(-2\delta, 2 \delta)$.
\end{proof}
\begin{Lemma}
\label{lem4.4}
Under the conditions of Proposition \ref{prop0}, 
for $B_{m,n}$ with support in $(-2\delta, 2 \delta)$ if the integral on the right side in
\eqref{eqlem4.4} exists, then
\begin{multline}
\label{eqlem4.4}
\int_{\mathbb{R}} \omega_{m,n} (\tau; b)\Big | \mathcal{F} \left [ 
\phi (t) \int_{-2\delta}^0 e^{i l_{m,n} (t-s) -\alpha n^2 (|t|-s)} B_{m,n} (s) ds \right ] (\tau) \Big |^2 d\tau 
\\
\lesssim \int_{\mathbb{R}} \omega_{m,n} (\tau; b) 
\Big | \mathcal{F} \left [ \phi (t) \int_{\mathbb{R}} H (t-s) e^{[i l_{m,n} -\alpha n^2 ] (t-s)} 
B_{m,n} (s) ds \right ] \Big |^2 (\tau) d\tau  
\end{multline}
\end{Lemma}
\begin{proof}
We note from Lemma \ref{lem4.1} and \ref{lem4.3}, that
\begin{equation}
\begin{aligned}
\label{28}
\int_{\mathbb{R}} \omega_{m,n} (\tau; b)&\Big | \mathcal{F} \left [ 
\phi (t) \int_{-2\delta}^0 e^{i l_{m,n} (t-s) -\alpha n^2 (|t|-s)} B_{m,n} (s) ds \right ] (\tau) \Big |^2 d\tau 
\\
&\lesssim
\left [ \int_{\mathbb{R}} \omega_{m,n} (\tau; b)\Big | \mathcal{F} \left [ e^{i l_{m,n} t - \alpha n^2 |t|} \right ] (\tau) \Big |^2 d\tau \right ] 
\Big | \int_{-2\delta}^0 e^{-i l_{m,n} s +\alpha n^2 s} B_{m,n} (s) ds \Big |^2 \\ 
&\lesssim \left ( \int_{\mathbb{R}} \frac{\omega_{m,n} (\tau; b)n^4}{\left (n^2+|\tau-l_{m,n}| \right )^4} d\tau \right )
\Big | \int_{\mathbb{R}} H(-s) e^{-i l_{m,n} s +\alpha n^2 s} B_{m,n} (s) ds \Big |^2  \\
&\lesssim \int_{\mathbb{R}} \frac{n^4 d\tau}{\left (n^2+|\tau-l_{m,n}| \right )^{4-2b}} 
\sup_{t} \Big | 
\phi(t)\int_{\mathbb{R}} H(t-s) e^{i l_{m,n} (t-s) -\alpha n^2 (t-s)} B_{m,n} (s) ds \Big |^2  \\
&\lesssim 
\left ( \int_{0}^\infty \frac{|n|^{6-4b} d\tau}{(n^2+\tau)^{4-2b}} \right )
\int_{\mathbb{R}} \left ( n^2 + |\tau-l_{m,n}| \right )^{2b}  \\
&\times \Big | \mathcal{F} \left [\phi(t)  
\int_{\mathbb{R}} H(t-s) e^{i l_{m,n} (t-s) -\alpha n^2 (t-s)} B_{m,n} (s) ds \right ] (\tau)  \Big |^2 d\tau \\ 
\end{aligned}
\end{equation}
from which Lemma follows readily.
\end{proof}
\begin{Lemma} 
\label{lemextra}
Assume that conditions of Proposition \ref{prop0} hold and that
$\boldsymbol{\theta}, \boldsymbol{\beta} \in T_{s,b}$.
Define
\begin{equation}
d_{m,n} = e^{i l_{m,n} t - \alpha n^2 |t|} \int_{\mathbb{R}} H(-s) e^{(\alpha n^2 - i l_{m,n} )s}
e^{\beta s} \phi^2 (s) \Lambda_{m,n} (s) ds \ ,
\end{equation} 
where
\begin{equation}
\Lambda_{m,n} (t) = -i m\sum_{(m',n') \in \mathbb{Z}^2_{0,n}}
\theta_{m',n'} (t) \beta_{m-m', n-n'} (t) . 
\end{equation}
Then
\begin{equation}
\sum_{(m,n) \in \mathbb{Z}^ 2_0}
 \int_{\mathbb{R}} W_{m,n} (\tau; b) \Big | \mathcal{F} [ \phi d_{m,n}] (\tau) \Big |^2 d\tau
\lesssim \delta^{2 b'-2b} 
\| \boldsymbol{\theta} \|_{s,b}^2
\| \boldsymbol{\beta} \|_{s,b}^2
\end{equation}
\end{Lemma}
\begin{proof}
We apply Lemma
\ref{lem4.4} 
$$
B_{m,n} (t) = \phi^2 (t) e^{\beta t} \Lambda_{m,n} (t) =-i m
\sum_{(m'n')\in \mathbb{Z}^2_{0,n}} {\hat \theta}_{m',n'} (t) {\hat \beta}_{m-m',n-n'} (t) 
$$     
where ${\hat \theta}_{m',n'} (t) = e^{\beta t/2} \phi (t) \theta_{m',n'} (t)$
and ${\hat \beta}_{m-m',n-n'} (t) = e^{\beta t/2} \phi (t) \beta_{m-m',n-n'} (t)$, 
and then use Lemma  
\ref{lem4.5}.
\end{proof}
\begin{Proposition}
\label{propM}
Under the conditions of Proposition \ref{prop0}, except for $b' > b$,
for $\mathcal{M}$ defined in \eqref{1.8}-\eqref{1.10},
there exists constant $c_1$ independent of $\delta$ such that
for $\boldsymbol{\theta}^{(1)}, \boldsymbol{\theta}^{(2)} \in T_{s,b}$,
\begin{equation}
\| \mathcal{M} [\boldsymbol{\theta}^{(1)} ] 
- \mathcal{M} [\boldsymbol{\theta}^{(2)} ] \|_{s,b} 
\le c_1 \delta^{b'-b} \| \boldsymbol{\theta}^{(1)} + \boldsymbol{\theta}^{(2)} \|_{s,b} 
\| \boldsymbol{\theta}^{(1)} - \boldsymbol{\theta}^{(2)} \|_{s,b} 
\end{equation}
\end{Proposition}
\begin{proof}
We note from definition in \eqref{1.8}-\eqref{1.10} 
\begin{multline}
\left [ \mathcal{M} [\boldsymbol{\theta}^{(1)} ] 
- \mathcal{M} [\boldsymbol{\theta}^{(2)}] \right ]_{m,n} (t) \\
= 
\phi (t) \int_{\mathbb{R}} 
H(t-s) e^{-(\alpha n^2 -i l_{m,n} ) (t-s) } e^{\beta s} \phi^2 (s) \Lambda_{m,n} (s) ds  
\\
-\phi (t) e^{i l_{m,n} t - \alpha n^2 |t|} \int_{\mathbb{R}} 
H(-s) e^{(\alpha n^2 -i l_{m,n} ) s } e^{\beta s} \phi^2 (s) \Lambda_{m,n} (s) ds \ ,  
\end{multline}
where 
\begin{equation}
\Lambda_{m,n} (t) =-i m \sum_{(m',n') \in \mathbb{Z}^2_{0,n}} \left ( \theta^{(1)}_{m',n'} (t) 
+ \theta^{(2)}_{m',n'} (t) \right ) \left ( \theta^{(1)}_{m-m',n-n'} (t)- 
\theta^{(2)}_{m-m',n-n'} (t) \right )
\end{equation}
the result follows from 
applying Lemmas \ref{lem4.5} and \ref{lemextra}, with $\boldsymbol{\theta} = \boldsymbol{\theta}^{(1)} + 
\boldsymbol{\theta}^{(2)}$, $\boldsymbol{\beta} = \boldsymbol{\theta}^{(1)} - 
\boldsymbol{\theta}^{(2)} $.
\end{proof}

\noindent{\bf Proof of Theorem \ref{Thm1}:}

The existence and uniqueness of solution in some ball containing the initial condition
by applying Proposition \ref{propM}
with $b ' > b$ (with all other given restriction) and Lemma \ref{lemIC},
which for $\boldsymbol{\theta}^{(1)} =\boldsymbol{\theta}$ and
$\boldsymbol{\theta}^{(2)} = \boldsymbol{0}$ (note $\mathcal{M} [\boldsymbol{0}]=0$) implies 
$\| \mathcal{M} [ \boldsymbol{\theta} ]  \| \le c_1 \delta^{b'-b} \|\boldsymbol{\theta} \|_{s,b}^2 $. Thus
in a ball of size
$2  c_0 \| u_0 \|_{H^s_{0,b}} $ in $T_{s,b}$  
for sufficiently small $\delta$ results in
$$
\| \mathcal{N} [\boldsymbol{\theta}] \|_{s,b} \le c_0 \|u_0 \|_{H^{s}_{0,b}} 
+ 4 c_0^2 c_1 \delta^{b'-b} \| u_0 \|^2_{H^{s}_{0,b}}
\le 2 c_0 \|u_0 \|_{H^{s}_{0,b}} $$      
$$
\| \mathcal{N} [\boldsymbol{\theta}^{(1)}] 
-\mathcal{N} [\boldsymbol{\theta}^{(2)}] \|_{s,b}  
\le 
4 c_0 c_1 \delta^{b'-b} \| u_0 \|_{H^{s}_{0,b}}
\| \boldsymbol{\theta}^{(1)} - \boldsymbol{\theta}^{(2)} \|_{s,b} \le \epsilon_1
\| \boldsymbol{\theta}^{(1)} - \boldsymbol{\theta}^{(2)} \|_{s,b} 
$$
where $\epsilon_1 < 1$,
and therefore $\mathcal{N}$ is contractive in the ball.
Further, for any time $t \in [0, \delta]$, applying Lemma \ref{lem4.3} it follows that
for $s > \frac{5}{2}$,
$$ \| u (\cdot, \cdot, t) \|_{H_{0,b}^s}^2 
\lesssim \sum_{(m,n) \in \mathbb{Z}^2_0} W_{m,n} (\tau) 
\Big | \mathcal{F} [ \theta_{m,n} ] (\tau)  \Big |^2 d\tau = \| \boldsymbol{\theta} \|_{s,b}^2 $$ 
Furthermore the estimate on $\| \mathcal{M} [\boldsymbol{\theta}] \|_{s,b}$ 
implies from the integral equation itself that
$$ \| u (\cdot, \cdot, t) - u^{(0)} (\cdot, \cdot, t) \|_{H_{0,b}^s} \lesssim
\| \boldsymbol{\theta} - \boldsymbol{\theta}^{(0)} \|_{s,b} = \| \mathcal{M} [ \boldsymbol{\theta} ] \|_{s,b} 
\lesssim \delta^{b'-b} \| \boldsymbol{\theta} \|^2_{s,b} $$
which by shrinking $\delta$ implies time continuity of solution at $t =0$ in $H_{0,b}^s$ since
$u^{(0)} (\cdot, \cdot, t)$ is obviously continuous. 
To prove continuity at a point $t_0 \in [0, \delta]$,
we rewrite (\ref{1.6}) in the form 
$$ u_{m, n} = u_{m,n} (t_0) + 
\int_{t_0}^t e^{-(n^2 -i l_{m,n} ) (t-s)} e^{\beta s} A_{m,n} (s) ds \ ,$$
which is clearly possible, 
and find an equivalent integral equation similar to (\ref{1.8}), except centering it at $t_0$. 
It is clear that the same argument shows continuity at $t_0$. 
that is left is the global energy control and 
uniqueness argument for solution to (\ref{1.6}). This is accomplished
through a more traditional energy type argument in the following subsection.

\subsection{Energy control, Uniqueness and end of Theorem \ref{Thm1} proof}

While the argument in the ensuing can be carried out in the space of sequences $\left \{ u_{m,n} \right \}$
by working directly with (\ref{1.6}), and doing inner product through in $(m,n) \in \mathbb{Z}^2_0$, it is
easier to see that this argument directly follows\footnote{While the manipulation to get energy inequality is
formal in the sense
that the derivatives in $x$ and $y$ are not assured, the end product is legitimate as it involves norms
that exist.}
from (\ref{1.1}-\ref{1.3}). 
It is convenient to introduce
\begin{equation}
\label{u.1}
v (x,y,t) = \int_0^y u(x,y',t) dy' 
\end{equation}
It is to be noted that $v \in \mathbb{T} [0, 2\pi]^2$ since $\int_0^{2\pi} u(x,y,t) dy=0$.
Then on integrating (\ref{1.1}) from $0$ to $y$, noting $u_y=0$ and $u_{yyy}=0$ at $y=0$, we obtain
\begin{equation}
\label{u.2}
v_{yyt} -\gamma v_{xxx}-\alpha v_{yyyy}-\beta v_{yy} + \left ( v_y^2 \right )_{xy} = 0 
\end{equation}
Using inner product of (\ref{u.2}) with $v$ and appropriate integration by parts, 
we obtain the energy bound
\begin{equation}
\label{u.3}
\frac{d}{dt} \frac{1}{2} \| u \|_{L^2}^2 + \alpha \| u_y \|_{L^2}^2 - \beta \| u \|_{L^2}^2 = 0 
\end{equation} 
It follows
\begin{equation}
\label{u.3.0}
\frac{1}{2} \| u (\cdot, \cdot, t) \|_{L^2}^2 
+ \alpha \int_0^t  \| u_y (\cdot, \cdot, s) \|_{L^2}^2 
ds
= 
\frac{1}{2} \| u_0 \|_{L^2}^2 +   
\beta \int_0^t \| u (\cdot, \cdot, s) \|_{L^2}^2 
d\tau
\end{equation}
Using Poincare inequality and Gronwall's Lemma, it follows that
\begin{equation}
\label{u.3.1}
\frac{1}{2} \| u (\cdot, \cdot, t) \|_{L^2}^2 
\le \frac{1}{2} \| u_0 \|_{L^2}^2 e^{2 (\beta-\alpha) t} 
\end{equation}
and we have global exponential control of $L^2$ norm.
If ${\tilde u}$ is another solution, and we define corresponding ${\tilde v} = \int_0^y {\tilde u} dy$,
Then, ${\tilde v}$ satisfies (\ref{u.2}) as well. 
Subtracting, we obtain the following equation for $w=v-{\tilde v}$:
\begin{equation}
\label{u.4}
w_{yyt} -\gamma w_{xxx} -\alpha w_{yyyy} -\beta w_{yy} + \left [ \left (v_y + {\tilde v}_y\right ) w_y \right ]_{xy} = 0
\end{equation} 
Inner product with $w$ in (\ref{u.4}), integration by parts and then time integration leads to
\begin{equation}
\label{u.5}
\frac{1}{2} \| w_y (\cdot, \cdot, t) \|_{L^2}^2 
+ \int_{0}^t \alpha \| w_{yy} (\cdot, \cdot, s)\|_{L^2}^2 d\tau 
=\beta \int_0^t \| w_y (\cdot, \cdot, s) \|_{L^2}^2 
- \frac{1}{2} \int_0^t \left ( w_{y}^2, [ v_y + {\tilde v}_y ]_{x}  \right ) (s) ds   
\end{equation}
Recalling $w_y=u-{\tilde u}$ and $v_y =u$, ${\tilde v}_y = {\tilde u}$,
using Gronwall's Lemma in the integral form, it follows $u = {\tilde u}$
and the solution to (\ref{1.6}) is unique
in $L^2 \left ( \mathbb{T} [0, 2 \pi ]^2 \right )$ within the class of functions for which
\begin{equation}
\label{u.5.1}
\int_0^t \| u_x (\cdot, \cdot, s) \|_{\infty} ds < \infty  
\end{equation}
From Sobolev embedding theorem, it is enough to require
$\| u (\cdot, \cdot, t) \|_{H^{s}_0}$ 
for $s > 2$
is time integrable. 
This is certainly true for the $u$ shown to exist for
$t \in [0, \delta]$ because the 
corresponding  
$\boldsymbol{\theta} \in T_{s,b}$ to (\ref{1.8}) for $s > 5/2$.
This completes the proof of Theorem \ref{Thm1}.

\begin{Remark}
{\rm 
We have made no attempts to optimize in $s$ or $b$; in all likelihood, the
solution exists in weaker spaces.}
\end{Remark}

\section{Estimates and proof of Proposition \ref{prop0}}\label{prelm}

\begin{Definition}
For given $(m,n) \in \mathbb{Z}^2_0$, define
\begin{equation}
\mathbb{Z}^{2}_{0,n, >} = \left \{ (m',n') \in \mathbb{Z}^2_{0,n}: n'/n \ge \frac{1}{2} \right \}
\end{equation}
Also, define
\begin{equation}
Q_{m,n,m',n'} = |n|^{-4b} \left ( n^2 + \Big | l_{m,n} - l_{m',n'} - l_{m-m',n-n'} \Big | \right )^{2b'-2}
\end{equation}
\end{Definition}
\begin{Lemma}
\label{lem0}
\begin{multline}
\label{eqlem0}
\sum_{(m',n') \in 
\mathbb{Z}^2_{0,n}} 
\int_{\mathbb{R}} \frac{m^2 W_{m,n} (\tau; b')d\tau_1}{\left ( n^2 + |\tau-l_{m,n}| \right )^2 
W_{m',n'} (\tau_1; b) W_{m-m',n-n'} (\tau-\tau_1; b)} 
\\
\lesssim 
\sum_{(m',n') \in 
\mathbb{Z}^2_{0,n, >}} 
\int_{\mathbb{R}} \frac{m^2 W_{m,n} (\tau; b')d\tau_1}{\left ( n^2 + |\tau-l_{m,n}| \right )^2 
W_{m',n'} (\tau_1; b) W_{m-m',n-n'} (\tau-\tau_1; b)} 
\end{multline}
\end{Lemma}
\begin{proof}
We simply note that the summand above is invariant
on change of variables $(m',n', \tau_1) \rightarrow (m-m',n-n',\tau-\tau_1)$; hence
\begin{multline}
\sum_{(m',n') \in 
\mathbb{Z}^2_{0,n}, | n'| \le |n-n'|} 
\int_{\mathbb{R}} \frac{m^2 W_{m,n} (\tau; b')d\tau_1}{\left ( n^2 + |\tau-l_{m,n}| \right )^2 
W_{m',n'} (\tau_1; b) W_{m-m',n-n'} (\tau-\tau_1; b)} 
\\
= \sum_{(m',n') \in 
\mathbb{Z}^2_{0,n}, |n'| \ge |n-n'|} 
\int_{\mathbb{R}} \frac{m^2 W_{m,n} (\tau; b')d\tau_1}{\left ( n^2 + |\tau-l_{m,n}| \right )^2 
W_{m',n'} (\tau_1; b) W_{m-m',n-n'} (\tau-\tau_1; b)}  \ ,
\end{multline}
Hence the Lemma follows since the set $|n'| \ge |n'-n|$ is equivalent to $n'/n \ge \frac{1}{2}$.
\end{proof}
\begin{Definition}
\label{defB0}
For any $(m,n) \in \mathbb{Z}_0^2$, define the corresponding set
$$ \mathcal{B}_0 = \left \{ (m',n') \in \mathbb{Z}^2_{0,n,>}, n'=2n \right \} $$
\end{Definition}
\begin{Remark}
We first consider the summation over this special set $\mathcal{B}_0 \subset \mathbb{Z}^2_{0,n,>}$, since
this analysis differs substantially from the rest of the summation set $\mathbb{Z}^{2}_{0,n,>}$.
\end{Remark}
\begin{Lemma}
\label{lemrhosup}
For any $s \ge 0$,
\begin{equation}
\label{eqlemrhosup}
\sup_{(m',n') \in \mathcal{B}_0} 
\frac{\rho_{m,n}}{\rho_{m',n'} \rho_{m-m',n-n'}}  \lesssim 1 
\end{equation} 
\end{Lemma}
\begin{proof}
We note from expression of $\rho_{m,n}$ that  
\begin{equation}
\frac{\rho_{m,n}}{\rho_{m',2n} \rho_{m-m',n}}  
= \frac{(|m|+|n|)^{2s}}{\left ( |m'| + 4 |n| \right )^{2s} \left ( |m-m'| + |n| \right )^{2s}}
\lesssim 1 \ ,
\end{equation}
since 
we can break up the set $\mathcal{B}_0$ into $|m'| \le |m|/2$ and its complement $|m'| > |m|/2$,    
and the expression above is bounded in either case.
\end{proof}
\begin{Lemma}
\label{lemB0}
For $|m| \ge 1$, under conditions of Proposition \ref{prop0},
\begin{equation}
\label{eqlemB0}
\sum_{(m',n') \in \mathcal{B}_0} 
\int_{\mathbb{R}} \frac{m^2 W_{m,n} (\tau; b')d\tau_1}{\left ( n^2 + |\tau-l_{m,n}| \right )^2 
W_{m',n'} (\tau_1;b) W_{m-m',n-n'} (\tau-\tau_1;b)} 
\lesssim 1
\end{equation}
\end{Lemma}
\begin{proof}
Using expressions for $W_{m,n} (\tau)$, changing integration variable $\tau_1 \rightarrow \tau-\tau_1$ 
and using Lemma \ref{lemrhosup}, it is enough to show
\begin{equation}
\label{eqlemB01}
\sum_{m' \in \mathbb{Z}}
\int_{\mathbb{R}} \frac{m^2 \omega_{m,n} (\tau; b'-1) d\tau_1}{\omega_{m',2n} (\tau-\tau_1; b) \omega_{m-m',-n} (\tau_1; b) } \lesssim  1
\end{equation}
Since we are only interested in determining supremum of the left hand side of (\ref{eqlemB01}),
it is convenient to change variables $(\tau, \tau_1) \rightarrow (\tau n^2, \tau_1 n^2 )$, in which
case, we get the above to simplify to 
\begin{multline}
\frac{1}{|n|^{8b-4b'+2}} \sum_{m' \in \mathbb{Z}}
\int_{\mathbb{R}} \frac{m^2 \left(1 + |\tau-\frac{\gamma m^3}{n^4}| \right)^{2b'-2} d\tau_1}{ 
\left (4 + |\tau-\tau_1-\frac{\gamma {m'}^3}{n^4}| \right )^{2b}  
\left (1 + |\tau_1-\frac{\gamma (m-m')^3}{n^4}| \right )^{2b} } \\
\lesssim
\frac{1}{|n|^{8b-4b'+2}} \sum_{m' \in \mathbb{Z}}
\int_{\mathbb{R}} \frac{m^2 \left (1 + |\tau-\frac{\gamma m^3}{n^4}| \right )^{2b-2} d\tau_1}{ 
\left (1 + |\tau-\tau_1-\frac{\gamma {m'}^3}{n^4}| \right )^{2b}  
\left (1 + |\tau_1-\frac{\gamma (m-m')^3}{n^4}| \right )^{2b} } 
\end{multline}
We may assume $\gamma > 0$ without loss of generality, 
as otherwise for $\gamma < 0$, we replace $(\gamma, \tau) \rightarrow (-\gamma, -\tau)$ in the argument
given in the ensuing. 
Introducing $\xi = \gamma^{1/3} m/|n|^{4/3} $, 
$\xi_1 = |\gamma|^{1/3} m'/|n|^{4/3} $ we obtain by using Lemma \ref{lem4.2.c} that the above expression
\begin{multline}
\lesssim \frac{1}{|n|^{8b-4b'-2}} 
\int_{\mathbb{R}} \int_{\mathbb{R}} \frac{ \xi^2 \left ( 1 + |\tau-\xi^3| \right )^{2b-2} d\tau_1 d\xi_1}{
\left ( 1 + |\tau-\tau_1-\xi_1^3| \right )^{2b}
\left ( 1 + |\tau_1-(\xi-\xi_1)^3| \right )^{2b}} \\
\lesssim \frac{1}{|n|^{8b-4b'-2}} 
\int_{\mathbb{R}} \frac{\xi^2 \left ( 1 + |\tau-\xi^3| \right )^{2b-2} d\xi_1}{
\left ( 1 + |\tau- \xi_1^3 - (\xi-\xi_1)^3 | \right )^{2b}} 
\end{multline}
Using Lemma \ref{lem4.6} in the Appendix (originally due to Koenig, Ponce \& Pega\cite{KoenigPV}), 
the above 
\begin{equation}
\lesssim \frac{\xi^2}{\sqrt{|\xi|} \left ( 1 + |4 \tau - \xi^3| \right )^{1/2} 
\left ( 1 + |\tau-\xi^3| \right )^{2-2b'} } \lesssim 1  
\end{equation}
\end{proof}
\begin{Remark}
Having got the sum over the set $\mathcal{B}_0$ out of the way, we now consider the rest. For that purpose
it is useful to reduce the integration over $\tau_1$ into a simpler expression as in the following Lemma.
\end{Remark}  
\begin{Lemma}
\label{lem1}
For $(m,n) \in \mathbb{Z}_0^2 $, under conditions of Proposition \ref{prop0},
we have
\begin{multline}
\label{4}
\sup_{\tau \in \mathbb{R}} 
\sum_{(m',n') \in \mathbb{Z}_{0,n}}  
\int_{\mathbb{R}} \frac{m^2 W_{m,n} (\tau; b')d\tau_1}{\left ( n^2 + |\tau-l_{m,n}| \right )^2 
W_{m',n'} (\tau_1; b) W_{m-m',n-n'} (\tau-\tau_1; b)} \\
\lesssim 
\sum_{(m',n') \in \mathbb{Z}^{2}_{0,n,>}}  \frac{m^2 \rho_{m,n} Q_{m,n,m',n'}}{\rho_{m-m',n-n'} \rho_{m',n'}} 
\end{multline}
\end{Lemma}
\begin{proof}
We recall $W_{m,n} (\tau; b) = \omega_{m,n} (\tau; b) \rho_{m,n}$ and from definition of $\omega_{m,n}$,
$$
\frac{\omega_{m,n} (\tau; b')}{\left ( n^2 + |\tau-l_{m,n}| \right )^2 } = 
\omega_{m,n} (\tau; b'-1) =
\frac{1}{\omega_{m,n} (\tau; 1-b')}$$
First we note that if $|n'| \ge |n-n'|$, which is equivalent to $n'/n \ge 1/2$, 
Lemma \ref{lem4.2.c} with 
$k_1=(n-n')^2$, $k_3={n'}^2$, $k_4=l_{m',n'}$, $k_2=l_{m-m',n-n'}$,
implies
\begin{equation}
\label{5}
\int_{\mathbb{R}} \frac{\omega_{m,n} (\tau; b'-1) d\tau_1}{\omega_{m',n'} (\tau_1; b) 
\omega_{m-m',n-n'} (\tau-\tau_1; b)} 
\lesssim \frac{1}{|n-n'|^{4b-2}  \omega_{m,n} (\tau; 1-b') \omega_{m',n'} (\tau-l_{m-m',n-n'}; b)}  
\end{equation}
Therefore, 
\begin{multline}
\label{6}
\sum_{(m',n') \in \mathbb{Z}_{0,n,>}^2}  
\int_{\mathbb{R}} \frac{m^2 W_{m,n} (\tau; b')d\tau_1}{ \left ( n^2 + |\tau-l_{m,n} | \right )^2 
W_{m',n'} (\tau_1; b) W_{m-m',n-n} (\tau-\tau_1; b) }
\\ \lesssim 
\sum_{(m',n') \in \mathbb{Z}_{0,n, >}^2}  
\frac{m^2 \rho_{m,n}}{ \rho_{m',n'} \rho_{m-m',n-n'} 
|n-n'|^{4b-2}  \omega_{m,n} (\tau; 1-b') \omega_{m',n'} (\tau-l_{m-m',n-n'}; b)}  
\end{multline}
Further by switching variables $(n', m', \tau_1) \rightarrow (n-n', m-m', \tau-\tau_1)$, it follows that 
\begin{multline}
\label{7}
\sum_{(m',n') \in \mathbb{Z}_{0,n}^2, |n'| < |n-n'|}  
\int_{\mathbb{R}} \frac{m^2 W_{m,n} (\tau; b')d\tau_1}{ \left ( n^2 + |\tau-l_{m,n} | \right )^2 
W_{m',n'} (\tau_1; b) W_{m-m',n-n'} (\tau-\tau_1; b) } \\
\lesssim
\sum_{(m',n') \in \mathbb{Z}_{0,n, >}}  
\int_{\mathbb{R}} \frac{m^2 \rho_{m,n} \omega_{m,n} (\tau; b'-1) d\tau_1}{\rho_{m',n'} \rho_{m-m',n-n'} \omega_{m',n'} (\tau_1; b)
\omega_{m-m', n-n'} (\tau-\tau_1; b) }  \\
\lesssim \sum_{(m',n') \in \mathbb{Z}_{0,n,>}^2}  
\frac{m^2 \rho_{m,n}}{ \rho_{m',n'} \rho_{m-m',n-n'} 
|n-n'|^{4b-2} \omega_{m,n} (\tau; 1-b') \omega_{m',n'} (\tau-l_{m-m',n-n'}; b)}  
\end{multline}
We also note $n'/n \ge \frac{1}{2}$, 
implies ${n'}^2 \ge \frac{n^2}{4}$ and
so in that case applying Lemma \ref{lem4.2.a}
\begin{multline}
\label{8} 
\frac{1}{\omega_{m,n} (\tau; 1-b') \omega_{m',n'} (\tau-l_{m-m',n-n'}; b)} 
\\
\lesssim \frac{1}{\left ( |n|^2/4 + \Big | \tau-l_{m,n} \Big | \right )^{2 (1-b')}     
\left ( {n'}^2 + \Big | \tau-l_{m',n'}-l_{m-m',n-n'} \Big | \right )^{2 b} } \\    
\lesssim 
\frac{1}{\left ( |n|^2/4 + \Big | \tau-l_{m,n} \Big | \right )^{2 (1-b')}     
\left ( \frac{n^2}{4} + \Big | \tau-l_{m',n'}-l_{m-m',n-n'} \Big | \right )^{2b} }     
\lesssim Q_{m,n,m',n'}  \ , 
\end{multline}
where in the last step we used Lemma \ref{lem4.2.a}.
Therefore, using \eqref{6}-\eqref{8}, the lemma follows.
\end{proof}
\begin{Lemma}
\label{lem2}
For $s > 1$ 
\begin{equation}
\sum_{(m',n') \in \mathbb{Z}^2_{0,n, >}} 
\frac{\rho_{m,n}}{\rho_{m',n'} \rho_{m-m',n-n'}} \lesssim 1
\end{equation} 
\end{Lemma}
\begin{proof}
We define
$$P_1 := \left \{ (m',n') \in \mathbb{Z}^{2}_{0,n}: (m')^2 + (n')^2 \ge \frac{1}{4} (m^2 + n^2) \right \}$$
$$P_2 := \left \{ (m',n') \in \mathbb{Z}^{2}_{0,n}: (m-m')^2 + (n-n')^2 \ge \frac{1}{4} (m^2 + n^2) \right \}$$
It is clear\footnote{Note that $P_1 \cap P_2 \ne \{0\}$, but this has no
bearing on the proof}from geometry that $\mathbb{Z}^2_{0,n} = P_1 \cup P_2$. 
By changing indices $(m',n') \rightarrow (m-m', n-n')$, it is clear that $\sum_{(m',n') \in P_1}$ and 
$\sum_{(m',n') \in P_2}$ contribute equally.   
However, because of the equivalence of norms
\begin{multline}
\sum_{(m',n') \in \mathbb{Z}^2_{0,n}} \frac{\rho_{m,n}}{\rho_{m',n'} \rho_{m-m',n-n'}} \lesssim
\sum_{(m',n') \in P_1} \frac{\rho_{m,n}}{\rho_{m',n'} \rho_{m-m',n-n'}} 
\\
\lesssim
\left ( \sum_{(m',n') \in P_1} 
\frac{1}{\left ( |n-n'|+|m-m'| \right )^{2s}}  \right ) 
\lesssim \sum_{(m',n') \in \mathbb{Z}^2_{0,n}} 
\frac{1}{\left ( |n'|^2+|m'|^2 \right )^{s}}  \lesssim   
\int_{1}^\infty r^{1-2s} dr \lesssim 1 
\end{multline}
\end{proof}
\begin{Proposition}
\label{prop1}
Under conditions of Proposition \ref{prop0}, if
$ |m| < |n|^{2-2 (b'-b)}$, 
$$ \sup_{\tau \in \mathbb{R}} \sum_{(m',n') \in \mathbb{Z}^2_{0,n}} 
\int_{\mathbb{R}} \frac{m^2 W_{m,n} (\tau;b') d\tau_1}{\left ( n^2 + |\tau-l_{m,n}| \right )^2 
W_{m',n'} (\tau_1;b) W_{m-m',n-n'} (\tau-\tau_1;b)}  
\lesssim \frac{m^2}{n^{4-4 (b'-b)}} \le 1
$$
\end{Proposition}
\begin{proof}
This follows from applying Lemmas \ref{lem1}-\ref{lem2} and noting that 
$$ Q_{m,n,m',n'} 
\le \frac{1}{|n|^{4-4b'+4b}}$$
\end{proof}
\begin{Remark}{\rm Because of Proposition \ref{prop1}, in the remaining, 
we now only need to 
consider
$ |m| \ge n^{2-2(b'-b)}$ 
to complete the proof of Proposition \ref{prop0}. For that reason, we may
assume from this point onwards,  
$m \ne 0$. For each such $(m,n)$ we break up   
summation index set $\mathbb{Z}^{2}_{0,n, >}$     
into smaller sets, each of which requires a different argument. 
We note that the partition set for summation indices  
depends on $(m,n)$; however,
since the bound on contribution of each such set to the summation 
is found independent of $(m,n)$, the
final result does not depend on the partition. 
First consider the subset in
$\mathbb{Z}^2_{0,n,>}$ 
for which $|m'| > \frac{3}{2} |m|$ in the following Lemma.
}
\end{Remark}
\begin{Definition}\label{defB1}{\rm
For each $(m,n) \in \mathbb{Z}_{0,n}$ define corresponding set 
$$\mathcal{B}_1 := \left \{ (m',n') \in \mathbb{Z}^2_{0,n,>} : |m'| > \frac{3}{2} |m| \right \} $$
}
\end{Definition} 
\begin{Lemma}
\label{lem3}
For $ s > 2 $, $m \ne 0$, 
\begin{equation}
\sum_{(m',n') \in \mathcal{B}_1}
\frac{m^2 \rho_{m,n} Q_{m,n,m',n'}}{\rho_{m',n'} \rho_{m-m',n-n'}} \lesssim 1
\end{equation} 
\end{Lemma}
\begin{proof} 
We note that since $Q_{m,n,m',n'} \le \frac{1}{|n|^{4-4 (b'-b)}}$,
\begin{multline}
\sum_{(m',n') \in \mathcal{B}_1 }
\frac{m^2 \rho_{m,n} Q_{m,n,m',n'}}{\rho_{m',n'} \rho_{m-m',n-n'}}\\
\lesssim 
\frac{1}{|n|^{4-4b'+4b}} \sum_{|m'|> \frac{3}{2} |m|, n'/n > \frac{1}{2}}  
\frac{m^2 (|n|+|m|)^{2s}}{(|n|/2+|m'|)^{2s} 
(|n-n'|+|m-m'|)^{2s}} \\
\lesssim
\frac{m^2}{|n|^{4-4b'+4b}} \sum_{|m'| > \frac{3}{2} |m|} \int_{-\infty}^\infty \frac{dr}{(|n-r|+|m'|/3)^{2s}}\\
\lesssim \frac{m^2}{|n|^{4-4b'+4b}} 
\sum_{|m'| > \frac{3}{2} |m|} \frac{1}{(|n|+|m'|/3)^{2s-1}} \lesssim \frac{|n|^{4b'-4b}}{n^4 |m|^{2s-4}} 
\lesssim 1
\end{multline}
\end{proof}
\begin{Definition}
\label{defxyT}
{\rm 
We define $\eta =m'/m $, $\zeta=n'/n$ and 
\begin{equation}
f(\eta,\zeta) = \frac{(\eta-\zeta)^2}{\zeta^2(1-\zeta)^2} (2 \zeta-1) \left ( \eta - g(\zeta) \right ) \ , 
g(\zeta) = \frac{\zeta(2-\zeta)}{2\zeta-1}   
\end{equation}
For given $(m,n) \in \mathbb{Z}_0 \setminus \{ m=0 \}$, define
\begin{equation}
\mathcal{T} = \left \{ (m',n') \in \mathbb{Z}^2_{0,n, >} : 
|\eta| \le \frac{3}{2} ~, ~\Big | (2\zeta-1) \eta-(2\zeta-1) g(\zeta) \Big | \le k (n) ~\text{or},~   
|\eta-\zeta| \le k (n) \right \} \ ,
\end{equation}
where
\begin{equation}
k(n) = \min \left \{ \frac{1}{10}, |n|^{-2/3+2b'-2b} \right \}
\end{equation}
$\mathcal{T}^c $ will denote the 
complementary set in 
$\mathbb{Z}^2_{0,n} \cup \left \{ (m',n'): |m'| \le \frac{3}{2} |m|
\right \}$, {\it i.e.} 
\begin{equation}
\mathcal{T}^c = \left \{ (m',n') \in \mathbb{Z}^2_{0,n, >} : 
|\eta| \le \frac{3}{2} ~, ~\Big | (2\zeta-1) \eta-(2\zeta-1) g(\zeta) \Big | > k (n) ~\text{and},~   
|\eta-\zeta| > k (n) \right \} \ ,
\end{equation}
}
\end{Definition}
\begin{Lemma}
\label{lemfmin}
For $ (m',n') \in \mathcal{T}^c$, 
there exists nonzero constant $C_0$ so that 
$$ \Big | l_{m,n} - l_{m-m',n-n'} - l_{m,n} \Big | \ge \frac{C_0 k^3 |m|^3}{n^2} $$  
\end{Lemma}
\begin{proof}
Calculation shows 
$$
l_{m,n}-l_{m-m',n-n'} - l_{m',n'}
= \frac{\gamma m^3}{n^2} f(\eta,\zeta) 
$$
First we consider $ \frac{1}{2} \le \zeta \le 5$. 
In that case
$$ \frac{1}{|f(\eta,\zeta)|} = \Big | \frac{\zeta^2(1-\zeta)^2}{(\eta-\zeta)^2 
\left [ (2\zeta-1) \eta - (2\zeta-\zeta^2) \right ] } 
\Big | \le \frac{1}{k^3} \left ( |\zeta|^2 |1-\zeta|^2 \right )
\lesssim \frac{1}{k^3} $$
Now, if $\zeta > 5$, then since $|\eta| \leq \frac{3}{2} $, we have
$$ \frac{1}{|f(\eta,\zeta)|} \le \Big | \frac{\zeta^2(1-\zeta)^2}{(\zeta-\frac{3}{2})^2 
\left [ \zeta^2-2\zeta-\frac{3}{2} (2\zeta-1) \right ]} \Big | 
\lesssim 1 ,$$
from which the Lemma follows.
\end{proof}
\begin{Lemma}
\label{lemT}
Under conditions given in Proposition \ref{prop0}, if $|m| \ge |n|^{2-2b'+2b}$,
\begin{equation}  
\label{eqlemT}
\sum_{(m',n') \in \mathcal{T}^c} \frac{m^2 \rho_{m,n} Q_{m,n,m',n'}}{\rho_{m',n'} \rho_{m-m',n-n'}} \lesssim 1
\end{equation}
\end{Lemma}
\begin{proof}
First there is nothing to prove for $m =0$, since it is obvious left side is zero. So, we will assume $m \ne 0$.
We first prove
\begin{equation}
\label{Qb}
\sup_{(m',n') \in \mathcal{T}^c} m^2 Q_{m,n,m',n'} 
\lesssim 1 
\end{equation}
Applying Lemma \ref{lemfmin} the left side of 
\eqref{eqlemT} is 
bounded by a constant multiple of 
\begin{equation}
\frac{m^2}{|n|^{4b} \left (\frac{k^3 |m|^3}{n^2} \right )^{2(1-b')}} 
\lesssim 1 \ ,
\end{equation}
under given restriction on $b, b'$,
implying \eqref{Qb}. Further applying Lemma \ref{lem2}, the result easily follows. 
\end{proof}
\begin{Definition}
\label{defS}
For given $(m,n) \in \mathbb{Z}^2_0 \setminus \{ m=0 \}$, define
$$ \mathcal{S}_1 = 
\left \{ (m',n')\in \mathbb{Z}^2_{0,n,>}: 
|1-\eta| \le k_1 (n) \right \} \ , $$
$$ \mathcal{S}_0 = 
\left \{ (m',n')\in \mathbb{Z}^2_{0,n,>}: 
|\eta| \le k_1 (n) \right \} , $$
$$ k_1 (n) = \frac{1}{10 |n|} \ , $$
and $\mathcal{S} = \mathcal{S}_0 \cup \mathcal{S}_1 $.
We denote $\mathcal{S}^c$ to be its complement 
in the set $\mathbb{Z}^2_{0,n,>} \cap \{ (m',n'): |m'| \le \frac{3}{2} |m| \}$ \ , {\it i.e.}
$$
\mathcal{S}^c = \left \{ (m',n')\in \mathbb{Z}^2_{0,n,>}: 
\frac{3}{2} \geqslant |\eta| > k_1 (n), |1-\eta| > k_1(n)| \right \} $$
\end{Definition}

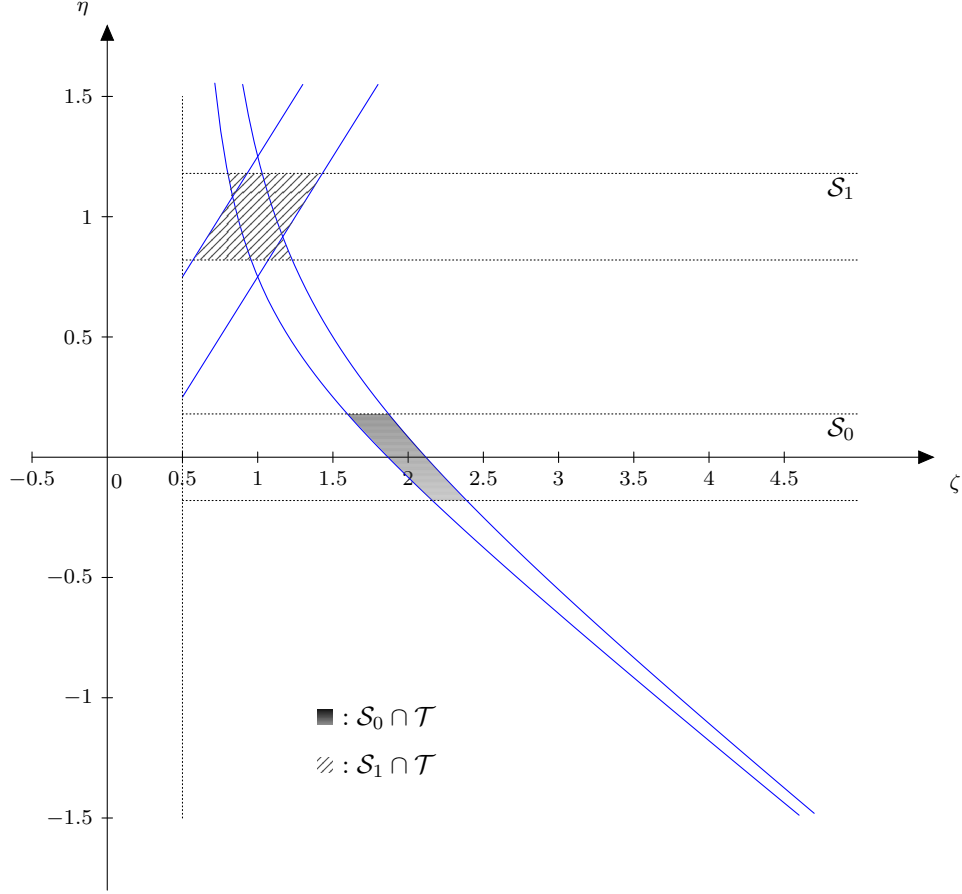
\begin{figure}
\begin{tikzpicture}[line cap=round,line join=round,>=triangle 45,x=2cm,y=3.2cm]
\draw[->,color=black] (-0.5,0) -- (5.5,0);
\foreach \x in {-0.5,0.5,1,1.5,2,2.5,3,3.5,4,4.5}
\draw[shift={(\x,0)},color=black] (0pt,2pt) -- (0pt,-2pt) node[below] {\footnotesize $\x$};
\draw[->,color=black] (0,-1.8) -- (0,1.8);
\foreach \y in {-1.5,-1,-0.5,0.5,1,1.5}
\draw[shift={(0,\y)},color=black] (2pt,0pt) -- (-2pt,0pt) node[left] {\footnotesize $\y$};
\draw[color=black] (-2pt,-9pt) node[right] {\footnotesize $0$};
\draw[color=black] (315pt,-10pt) node[right] {\footnotesize $\zeta$};
\draw[color=black] (-15pt,170pt) node[right] {\footnotesize $\eta$};
\draw[color=blue] plot[domain=0.5:1.3, range=-1.5:1.5] (\x,\x+0.25) node[right] {$$};
\draw[color=blue] plot[domain=0.5:1.8, range=-1.5:1.5] (\x,\x-0.25) node[right] {$$};
\draw [color=blue] plot[samples=100, domain=0.9:4.7, range=-1.5:1.5](\x,{0.25/(2*\x-1)+\x*(2-\x)/(2*\x-1)}) node[above left]{$$};
\draw [color=blue] plot[samples=100, domain=0.715:4.6, range=-1.5:1.5](\x,{-0.25/(2*\x-1)+\x*(2-\x)/(2*\x-1)}) node[above left]{$$};
\draw[-,densely dotted, color=black] (0.5,-1.5) -- (0.5,1.5);
\draw[densely dotted, color=black] plot[domain=0.5:5, range=-1.5:1.5] (\x,0.18) node[below left=-2pt] {$\mathcal{S}_0$};
\draw[densely dotted, color=black] plot[domain=0.5:5, range=-1.5:1.5] (\x,-0.18) node[right] {$$};
\draw[densely dotted, color=black] plot[domain=0.5:5, range=-1.5:1.5] (\x,1+0.18) node[below left=-2pt] {$\mathcal{S}_1$};
\draw[densely dotted, color=black] plot[domain=0.5:5, range=-1.5:1.5] (\x,1-0.18) node[right] {$$};
\shade[top color=black, bottom color=gray, opacity=0.3]
      (2.157,-0.179) -- (2.38,-0.179) -- (2.12,0) -- (1.87,0.179) -- (1.6,0.179) -- (1.87,0) -- cycle;
\fill[even odd rule,pattern=north east lines,pattern color=black!70,decoration=random steps]
      (0.577,1-0.179) -- (1.23,1-0.179) -- (1.165, 0.91) -- (1.43,1+0.179) -- (0.8,1.179) -- (0.825, 1.08) -- cycle;
\shade[top color=black, bottom color=gray, opacity=0.3]
      (1.4,-1.05) rectangle (1.5,-1.11);
\node [below right] at (+1.5,-1) {$: \mathcal{S}_0\cap \mathcal{T}$};
\fill[even odd rule,pattern=north east lines,pattern color=black!70,decoration=random steps]
      (1.4,-1.25) rectangle (1.5,-1.31);
\node [below right] at (+1.5,-1.2) {$: \mathcal{S}_1\cap \mathcal{T}$};
\end{tikzpicture}
\caption{$(\eta,\zeta)$-plane region corresponding to 
partition of $\mathbb{Z}^2_{0,n,>}\setminus \mathcal{B}_1$. $\mathcal{T}$ is the narrow
region between 
blue lines}\label{fig}
\end{figure}

\begin{Lemma}
\label{lemS}
Under conditions of Proposition \ref{prop0}, if $|m| \ge |n|^{2-2b'+2b}$,
\begin{equation}
\sum_{(m',n') \in \mathcal{S}^c} 
\frac{m^2 \rho_{m,n}Q_{m,n',m',n'}}{\rho_{m',n'} \rho_{m-m', n-n'} } \lesssim 
1
\end{equation}
\end{Lemma}
\begin{proof}
We note that
\begin{multline}
\sum_{(m',n') \in \mathcal{S}^c}
\frac{m^2 \rho_{m,n} Q_{m,n,m',n'}}{\rho_{m',n'} \rho_{m-m', n-n'} } 
\\
\lesssim  
\left \{ \int_{-\frac{3}{2}}^{-k_1}+\int_{k_1}^{1-k_1} + \int_{1+k_1}^{\frac{3}{2}} \right \} 
\int_{\frac{1}{2}}^{\infty} 
\frac{|m|^3 |n| Q_{m,n,m',n'}\rho_{m,n} d\zeta d\eta}{(1+|\eta| |m|+|\zeta| |n|)^{2s}
(1+|1-\eta| |m| + |1-\zeta| |n|)^{2s} }  \\
\lesssim \frac{|m|^3}{|n|^{3-4b'+4b}}
\left \{ \int_{-\frac{3}{2}}^{-k_1} + \int_{k_1}^{1-k_1} + \int_{1+k_1}^{\frac{3}{2}} \right \} \int_{\frac{1}{2}}^{\infty}
\frac{\rho_{m,n} d\zeta d\eta}{ (|n| + |\eta| |m| )^{2s}
(1+|1-\eta| |m| + |1-\zeta| |n|)^{2s}}  \\ 
\lesssim \frac{\rho_{m,n} |m|^3}{|n|^{4-4b'+4b}}
\left \{ \int_{-\frac{3}{2}}^{-k_1} + \int_{k_1}^{\frac{1}{2}}+\int_{\frac{1}{2}}^{1-k_1} + \int_{1+k_1}^{\frac{3}{2}} \right \}
\frac{d\eta}{
(|n|+|\eta| |m| )^{2s} (1+|1-\eta| |m| )^{2s-1}}  \\
\lesssim \frac{|n|^{4b'-4b}}{|n|^4 |k_1|^{2s-1} |m|^{2s-4}} \lesssim 1 \ ,
\end{multline}
under the restriction given
where in the last step we used the fact that in each of the four integration range, only one of the
$\rho_{m,n} (|n|+|\eta| |m|)^{-2s} \lesssim 1$ or $\rho_{m,n} (1+|1-\eta||m|)^{-2s} 
\lesssim 1$, while the remaining
factor in the integrand 
is $O((|k_1||m|)^{-2s+1})$
\end{proof}
\begin{Remark}
{\rm 
The previous two Lemmas give control over $\mathcal{S}^c\cup \mathcal{T}^c $. The complement set
$\mathcal{S}\cap \mathcal{T}$ has two components: $\mathcal{S}_1\cap \mathcal{T}$, 
that corresponds to a neighborhood of $(\eta,\zeta) = 
(1,1)$ and $\mathcal{S}_0 \cap \mathcal{T}$ that includes a
neighborhood of 
$(\eta,\zeta) = (0, 2)$.}
\end{Remark}
\begin{Lemma}
\label{lemS1T}
Under conditions of Proposition \ref{prop0}, if $|m| \ge |n|^{2-2b'+2b}$,
\begin{equation}
\sum_{(m',n')\in \mathcal{S}_1 \cap \mathcal{T}}  \frac{m^2 Q_{m,n,m',n'} \rho_{m,n}}{\rho_{m',n'} \rho_{m-m',n-n'}} 
\lesssim 1
\end{equation}
\end{Lemma}
\begin{proof}
We note for $(m',n') \in \mathcal{S}_1 \cap \mathcal{T}$, $(\eta,\zeta)$ is in some small neighborhood of
$(\eta,\zeta) = (1,1)$ (see top shaded region in Figure 1). 
For large $|n|$, The size of this region is $O (|n|^{-2/3+2b'-2b}) $ horizontally 
and $O(|n|^{-1})$ vertically.   
It is useful to write $(\eta,\zeta) = (1+\delta, 1+\Delta)$ and then it is clear that
$$ f(\eta,\zeta) = 
\frac{(\Delta-\delta)^2}{\zeta^2 \Delta^2} ( 1+2 \Delta) \left (\frac{\Delta (2+\Delta)}{1+2\Delta} +
\delta \right ) $$
Note $|\delta| \le 1/(10 |n|)$ and we can assume that for $(\eta,\zeta)$ with corresponding
$(m',n') \in \mathcal{S}_1 \cap \mathcal{T}$,
$|\Delta| < \frac{2}{9}$ and given the discreteness of $n'$,
$|\Delta| \ge \frac{1}{|n|}$; hence $|\delta| \le \frac{|\Delta|}{10}$. 
It is then clear from the above expression for $f$ that 
$ |f (\eta,\zeta) | \gtrsim |\Delta| \ge \frac{1}{|n|}$, implying that 
$$ m^2 Q_{m,n,m',n'}  \lesssim m^2 |n|^{-4b} \left ( n^2 + \frac{|m|^3}{|n|^3} \right )^{2b'-2}      
\lesssim \frac{|n|^{6-6b'-4b}}{|m|^{4-6b'}} \lesssim 1 $$
Therefore, from applying Lemma \ref{lem2}, it follows that  
$$
\sum_{(m',n')\in \mathcal{S}_1 \cap \mathcal{T}}  
\frac{m^2 Q_{m,n,m',n'} \rho_{m,n}}{\rho_{m',n'} \rho_{m-m',n-n'}} 
\lesssim 1 
$$
\end{proof}
\begin{Lemma}
\label{lemS0T}
Under conditions of Proposition \ref{prop0}, if $|m| \geqslant |n|^{2-2b'+2b}$, then
\begin{equation}
\sum_{(m',n')\in \mathcal{S}_0 \cap \mathcal{T}, n' \ne 2 n}  
\frac{m^2 Q_{m,n,m',n'} \rho_{m,n}}{\rho_{m',n'} \rho_{m-m',n-n'}} 
\lesssim 1
\end{equation}
\end{Lemma}
\begin{proof}
We note that for 
$(m',n')\in \mathcal{S}_0 \cap \mathcal{T}$, $n' \ne 2 n$, $(\eta,\zeta)$ 
is in a small neighborhood of $(0,2)$ not
including $(0, 2)$ itself, and for large $n$,
$|\eta| = O (|n|^{-1})$ and $|2-\zeta| = O (|n|^{-2/3+2b'-2b})$. 
It is convenient to write $
(\eta,\zeta) = 
\left ( \delta , 2+\Delta \right )$, 
and note that from discreteness of $n'$, $|\Delta| \ge \frac{1}{|n|}$.  
Then, 
we find
\begin{equation}
f(\eta,\zeta) = \frac{(2+\Delta-\delta)^2}{\zeta^2 (\zeta-1)^2} 
\left ( 2 \Delta + \delta (2 \Delta+3 ) + \Delta^2 \right ) 
\end{equation}
From the choice of $k_1 (n)$ and the $\frac{1}{|n|}$ lower bound of $|\Delta|$, it follows that
$|\delta | \le \frac{|\Delta|}{10}$, and therefore it is clear 
from above that $|f(\eta,\zeta)| \gtrsim |\Delta| \ge \frac{1}{|n|}$ and
thus $ Q_{m,n,m',n'} \lesssim |m|^{6b'-6} |n|^{6-6b'-4b}$. Therefore, it follows that 
$$ m^2 Q_{m,n,m',n'} \lesssim \frac{|n|^{6-6b'-4b}}{|m|^{4-6b'}}  \lesssim 1 $$
Therefore, from applying Lemma \ref{lem2}, it follows that  
$$
\sum_{(m',n')\in \mathcal{S}_0 \cap \mathcal{T}, n' \ne 2 n}  
\frac{m^2 Q_{m,n,m',n'} \rho_{m,n}}{\rho_{m',n'} \rho_{m-m',n-n'}} 
\lesssim 1 
$$
\end{proof}
\begin{Remark} Though the set 
$\left \{ (m',n') \in \mathcal{S}_0 \cap \mathcal{T} : n'=2n \right \}$ is not
covered by Lemma \ref{lemS0T}, we don't have to worry about this since control over the set $\mathcal{B}_0$
has already been shown in Lemma \ref{lemB0}.  
\end{Remark}
\noindent{\bf Proof of Proposition \ref{prop0}} now follows by applying
Lemma \ref{lem0} in \ref{lemB0}, 
and then using  Lemma \ref{lem1}
in 
Proposition \ref{prop1}, Lemmas \ref{lem3}, \ref{lemT}, \ref{lemS}, \ref{lemS1T}
and \ref{lemS0T}. In this context, it is useful to note that
$$\mathbb{Z}^{2}_{0,n, >} = \mathcal{B}_0 \cup \mathcal{B}_1 
\cup
\mathcal{S}^c \cup \mathcal{T}^c \cup \{ \mathcal{S}_1\cap \mathcal{T} \} \cup 
\left \{ \{ \mathcal{S}_0 \cap \mathcal{T} \} \setminus \mathcal{B}_0 \right \} \ ,
$$
where $\mathcal{B}_0$, $\mathcal{B}_1$, $\mathcal{S}^c$, $\mathcal{T}^c$, $\mathcal{T}$, $\mathcal{S}_0$ and
$\mathcal{S}_1$ are defined in Definitions \ref{defB0}, \ref{defB1}, \ref{defS} and \ref{defxyT}.

\section{Appendix}
\begin{Lemma}{(Based on Tao \cite{Tao}).}
\label{lem4.1}
If $\phi$ is any smooth function with support in $(-2\delta, 2\delta )$ and
$ \frac{1}{2} < b \le  b' < b + \frac{1}{2}$, with $b' < 1$, then 
\begin{equation}
\int_{\mathbb{R}}  \left ( n^2 + | \tau - l_{m,n} | \right )^{2b} \Big | \mathcal{F} [ \phi q] (\tau) 
\Big |^2 d\tau 
\lesssim \delta^{2b'-2b} \int_{\mathbb{R}} \left ( n^2 + | \tau-l_{m,n} ) | \right )^{2b'} \Big | 
\mathcal{F} [ q ] (\tau) \Big |^2 d\tau
\end{equation} 
\end{Lemma}
\begin{proof}
We define $\Phi (\tau) = \mathcal{F} [\phi] (\tau)$ and $Q(\tau) = \mathcal{F} [ q ] (\tau)$.
Then, we note that $\mathcal{F} [ \phi q ] (\tau) = \left [ Q*\Phi \right ] (\tau)$, where $~*~$ is
the Fourier-Convolution.
We decompose 
\begin{equation}
\label{4.2}
Q(\tau) = U^{(1)} (\tau) + U^{(2)} (\tau) \ , 
\end{equation}
where support of $U^{(2)}$ is in $n^2 + |\tau-l_{m,n}| \le \frac{1}{\delta}$ while support of $U^{(1)}$ is
in its complement.
Then 
\begin{multline}
\label{4.3}
\left \{ \int_{\mathbb{R}}  \left ( n^2 + | \tau - l_{m,n} | \right )^{2b} \Big | \mathcal{F} [ \phi q] (\tau) 
\Big |^2 d\tau \right \}^{1/2} \\
\lesssim  
\left \{ \int_{\mathbb{R}}  \left ( n^2 + | \tau - l_{m,n} | \right )^{2b} \Big | U^{(1)}*\Phi [\tau]  
\Big |^2 d\tau \right \}^{1/2} 
+\left \{ \int_{\mathbb{R}}  \left ( n^2 + | \tau - l_{m,n} | \right )^{2b} \Big | U^{(2)}*\Phi [\tau]  
\Big |^2 d\tau \right \}^{1/2} 
\end{multline}
Now, we note that the first term on the right is given by
\begin{equation}
\begin{aligned}
&\left \{ \int_{\mathbb{R}}  \left ( n^2 + | \tau - l_{m,n} | \right )^{2b} \Big |\int_{\mathbb{R}} U^{(1)} (\tau-\tau_1)
\Phi (\tau_1) d\tau_1  \Big |^2 d\tau \right \}^{1/2}
\\
&\lesssim  
\int_{\mathbb{R}} \left \{ \int_{\mathbb{R}}  
\left ( n^2 + | \tau -\tau_1 - l_{m,n} | \right )^{2b} \left ( 1+ |\tau_1| \right )^{2b} 
\Big | U^{(1)} (\tau-\tau_1) \Phi (\tau_1) \Big |^2 d\tau \right \}^{1/2} d\tau_1 \\
&\lesssim \int_{\mathbb{R}} \left \{ \int_{\mathbb{R}} 
\frac{\left ( n^2 + |\tau-\tau_1-l_{m,n} | \right )^{2b'}}{
\left ( n^2 + |\tau-\tau_1-l_{m,n} | \right )^{2(b'-b)}}
\Big | U^{(1)} (\tau-\tau_1)  \Big |^2 d\tau \right \}^{1/2} (1+|\tau_1|)^{b} 
|\Phi (\tau_1 )| d\tau_1    \\
&\lesssim \delta^{b'-b} 
\left \{ \int_{\mathbb{R}} 
\left ( n^2 + |\tau-\tau_1-l_{m,n} | \right )^{2b'}
\Big | U^{(1)} (\tau-\tau_1)  \Big |^2 d\tau \right \}^{1/2} 
\end{aligned}
\end{equation}
For the second part, note that 
\begin{equation}
\begin{aligned}
\int_{\mathbb{R}} & \left ( n^2 + | \tau - l_{m,n} | \right )^{2b} \Big | U^{(2)}*\Phi [\tau]  
\Big |^2 d\tau =
\int_{\mathbb{R}} \left ( n^2 + |\tau-l_{m,n} | \right )^{2b} \Big |\int_{\mathbb{R}} 
U^{(2)} (\tau_1)
\Phi (\tau-\tau_1) d\tau_1 \Big |^2 d\tau \\
&\lesssim   
\int_{\mathbb{R}} \Big | \int_{\mathbb{R}} \left ( n^2 + |\tau_1-l_{m,n} | \right )^{b} 
\left ( 1 + |\tau-\tau_1| \right )^{b} 
U^{(2)} (\tau_1)
\Phi (\tau-\tau_1) d\tau_1 \Big |^2 d\tau 
\\
&\lesssim 
\left \{ \int_{(n^2+|\tau_1-l_{m,n} |) \le \delta^{-1}} \frac{d\tau_1}{(n^2+|\tau_1-l_{m,n}|)^{2(b'-b)}} \right \} \\
&\times \left \{ \int_{\mathbb{R}} \int_{\mathbb{R}} |U^{(2)} (\tau_1)|^2 |\Phi (\tau-\tau_1) |^2 \left ( 1 + |\tau-\tau_1| \right )^{2b} 
\left ( n^2 + |\tau_1-l_{m,n} | \right )^{2b'} d\tau_1 d\tau \right \}  \\
&\lesssim \delta^{2b'-2b-1} \left ( \int_{\mathbb{R}} (1+{\hat \tau}^2) |\Phi ({\hat \tau}) |^2 d{\hat \tau} \right )
\left ( \int_{\mathbb{R}} \left ( n^2 + |\tau_1-l_{m,n} | \right )^{2b'} \Big | U^{(2)} (\tau_1) \Big |^2 d\tau_1 
\right )    
\\
&\lesssim 
\delta^{2b'-2b-1} \int_{-\delta}^{\delta} \left (|\phi (t)|^2 + |\phi^\prime (t) |^2 \right ) dt
\left ( \int_{\mathbb{R}} \left ( n^2 + |\tau_1-l_{m,n} | \right )^{2b'} \Big | U^{(2)} (\tau_1) \Big |^2 d\tau_1 
\right )   \\ 
&\lesssim \delta^{2b'-2b} 
\left ( \int_{\mathbb{R}} \left ( n^2 + |\tau_1-l_{m,n} | \right )^{2b'} \Big | Q (\tau_1) \Big |^2 d\tau_1 
\right )    
\end{aligned}
\end{equation}
Therefore the Lemma follows by combining the two results.
\end{proof}
\begin{Lemma}{(Based on Koenig, Ponce, Vega \cite{KoenigPV})}.
\label{lem4.2.a}
For $ 1> b > \frac{1}{2}$ and $\tau \in \mathbb{R}$,  
\begin{equation}
\frac{1}{\left (k_3 + |\tau-k_2| \right )^{2 (1-b')} \left ( k_3 + |\tau-k_4| \right )^{2b}}
\le \frac{1}{k_3^{2b} (k_3+|k_2-k_4|)^{2 (1-b')}}   
\end{equation}
and
\begin{equation}
\frac{1}{\left (k_3 + |\tau-k_2| \right )^{2 b} \left ( k_3 + |\tau-k_4| \right )^{2(1-b')}}
\le \frac{1}{k_3^{2b} (k_3+|k_2-k_4|)^{2 (1-b')}}   
\end{equation}
\end{Lemma}
\begin{proof}
Consider the first statement.
Note that $1-b' < b$. Note that we may write the left hand side as $\frac{1}{[f(\tau)]^{2(1-b')}}$ where
\begin{equation}
f(\tau) = \left (k_3 + |\tau-k_2| \right ) \left ( k_3 + |\tau-k_4 | \right )^{b/(1-b')}
\end{equation}
It is clear that in the intervals $\tau < \min \left \{ k_2, k_4 \right \}$, $ \tau > \max \left \{ k_2, k_4 \right \}$,
$f(\tau)$ is monotonically decreasing and increasing respectively with increasing $\tau$. 
Therefore, the minimum of $f(\tau)$ occurs for
$\tau \in [k_2, k_4]$ when $k_2 \leqslant k_4$ and for $\tau \in [k_4, k_2]$ when $k_2 > k_4$. 
Consider first the case when $k_2 \leqslant k_4$; then in $\tau \in [k_2, k_4]$, $f(\tau) = f_0 (\tau)$, where
\begin{equation}
f_0(\tau) = \left (k_3 + \tau-k_2 \right ) \left ( k_3 + k_4 - \tau \right )^{b/(1-b')}
\end{equation}
It can be checked that $f_0^\prime (k_4) < 0$ and simple calculus shows
$f_0^\prime (\tau) < 0$ for $\tau \in (\tau_c, k_4] $ and $f_0^\prime (\tau) > 0$ for
$\tau \in (-\infty, \tau_c )$ where
$$
\tau_c = \frac{-k_3 (b'+b-1) +k_4 (1-b') + k_2 b}{1-b'+b} < k_4 $$
We have two cases (i) $\tau_c < k_2$ and (ii) $\tau_c \in (k_2, k_4)$. In case (i), it is clear that the minimum
of $f_0(\tau)$ for $\tau \in [k_2, k_4]$ occurs at $\tau=k_4$. 
In case (ii), it is clear that $f_0(\tau)$ has a local
maximum at $\tau=\tau_c$. Therefore, in the interval $[k_2, k_4]$ minimum of $f_0$ 
is either at $k_2$ or $k_4$.
We notice that 
\begin{equation}
\frac{f_0(k_2)}{f_0(k_4)}  = 
\frac{k_3 \left ( k_3 + k_4 - k_2 \right )^{b/(1-b')}}{k_3^{b/(1-b')}  \left ( k_3 + k_4 -k_2 \right )}
= \left ( \frac{k_3+k_4-k_2}{k_3} \right )^{(b'+b-1)/(1-b')} > 1 
\end{equation}
implying that $f_0(k_2) > f_0(k_4)$. 
So, the minimum of $f$ in the interval $[k_2,k_4]$ occurs at $\tau=k_4$.
This implies that 
\begin{equation}
\frac{1}{\left (k_2 + |\tau-k_2| \right )^{2 (1-b')} \left ( k_3 + |\tau-k_4| \right )^{2b}}
= \frac{1}{[f(\tau)]^{2(1-b')}} \le  \frac{1}{[f(k_4)]^{2(1-b')}} =  
\frac{1}{k_3^{2b} (k_3+|k_2-k_4|)^{2 (1-b')}}   
\end{equation}
On the other hand if $k_2 > k_4$, in the interval $\tau \in [k_4, k_2]$, $f (\tau)=f_1 (\tau)$, where
\begin{equation}
f_1(\tau) = \left (k_3 + k_2 -\tau \right ) \left ( k_3 + \tau-k_4 \right )^{b/(1-b')}
\end{equation}
Then $f_1^\prime (\tau_c) = 0 $, where $\tau_c = k_3(b'+b-1) + k_4 (1-b')+k_2 b > k_4 $. Also, it is to be noted
that for $\tau \in (k_4, \tau_c)$, $f^\prime > 0$ and $f^\prime  < 0$ for $\tau \in (\tau_c, \infty)$. Therefore,
over the interval $[k_2, k_4]$ minimum can occur at either $k_2$ or $k_4$. However, since   
\begin{equation}
\frac{f_1(k_2)}{f_1(k_4)} =  \left ( \frac{k_3 +k_2-k_4}{k_3} \right )^{(b'+b-1)/(1-b')} > 1 
\end{equation}
Therefore, once again, the minimum of $f(\tau)$ is at $\tau=k_4$ as before and we have the same result and
the first statement has been proved. For the second part, we simply interchange $k_2$ and $k_4$ in the previous
argument and the result follows.
\end{proof}
\begin{Lemma}{(Based on Koenig, Ponce, Vega \cite{KoenigPV})}
\label{lem4.2.c}
For $ \frac{1}{2} < b < 1$ and $k_3 \ge k_1 \ge 1$, then the following hold 
\begin{equation}
\int_{\mathbb{R}} \frac{d\tau}{\left ( k_1 + |\tau-k_2 | \right )^{2b} \left ( k_3 + |\tau-k_4 | \right )^{2b}}
\lesssim \frac{1}{k_1^{2b-1} \left ( k_3 + |k_2-k_4| \right )^{2b} }
\end{equation}
\end{Lemma}
\begin{proof}
It suffices to argue only the case $k_2 \le  k_4$; 
since the final result is symmetric in $k_2$ and $k_4$;
if this were not true, we can simply
interchange the $k_2$ and $k_4$ in the following argument.
We note that 
\begin{equation}
\frac{1}{\left ( k_1 + |\tau-k_2 | \right )^{2b} \left ( k_3 + |\tau-k_4 | \right )^{2b}}
\le \frac{1}{(k_3 + |\tau-k_2|)^{2b} \left ( k_1 +|\tau-k_4| \right )^{2b}} 
\ , {\rm for} ~\tau \ge \frac{k_2+k_4}{2}
\end{equation}
Then, the integral becomes 
\begin{equation}
\begin{aligned}
&\left \{ \int_{-\infty}^{(k_2+k_4)/2} + \int_{(k_2+k_4)/2}^{\infty} \right \} 
\frac{d\tau}{\left ( k_1 + |\tau-k_2 | \right )^{2b} \left ( k_3 + |\tau-k_4 | \right )^{2b}}
\\
&\le 
\int_{-\infty}^{(k_2+k_4)/2} 
\frac{d\tau}{\left ( k_1 + |\tau-k_2 | \right )^{2b} \left ( k_3 + |\tau-k_4 | \right )^{2b}}
+\int_{(k_2+k_4)/2}^\infty 
\frac{d\tau}{\left ( k_3 + |\tau-k_2 | \right )^{2b} \left ( k_1 + |\tau-k_4 | \right )^{2b}}
\\
&\le 
\frac{1}{(k_3+|k_2-k_4|/2)^{2b}} \left\{\int_{-\infty}^{(k_2+k_4)/2} 
\frac{d\tau}{\left ( k_1 + |\tau-k_2 | \right )^{2b} }
+\int_{(k_2+k_4)/2}^\infty 
\frac{d\tau}{\left ( k_1 + |\tau-k_4 | \right )^{2b} }\right\} \\
&\le \frac{1}{(k_3+|k_2-k_4|/2)^{2b}} \left \{ \int_{-\infty}^{k_2}
\frac{d\tau}{\left ( k_1 - \tau+k_2  \right )^{2b} }
+ \int_{k_2}^{(k_2+k_4)/2} 
\frac{d\tau}{\left ( k_1 - k_2+\tau \right )^{2b} } \right . \\
&\left . + \int_{(k_2+k_4)/2}^{k_4}
\frac{d\tau}{\left ( k_1 - \tau+k_4 \right )^{2b} } 
+\int_{k_4}^{\infty}
\frac{d\tau}{\left ( k_1 - k_4+\tau \right )^{2b} } \right \} 
\lesssim \frac{1}{k_1^{2b-1} \left ( k_3+|k_2-k_4| \right )^{2b}}
\end{aligned}
\end{equation}
\end{proof}
\begin{Lemma}{(Based on Tao \cite{Tao}).}
\label{lem4.3}
If $b > \frac{1}{2}$, then 
\begin{equation}
\sup_{t} \Big | M (t) \Big |^2 \lesssim \frac{1}{|n|^{4 b-2}} \int_{\mathbb{R}} \left ( n^2 + |\tau-l_{m,n} | \right )^{2b}
\Big | \mathcal{F} [ M ] (\tau) \Big |^2 d\tau  
\end{equation} 
\end{Lemma}
\begin{proof}
We note that
\begin{equation}
M(t) = \int_{\mathbb{R}} e^{i \tau t} \mathcal{F} [ M] (\tau) d\tau
\end{equation}
With substitution $\tau = \tau_0 + l_{m,n}$, 
\begin{equation}
M(t) = e^{i l_{m,n} t} \int_{\mathbb{R}} e^{i \tau_0 t} \mathcal{F} [ M] (\tau_0+l_{m,n}) d\tau_0
\end{equation}
So, for any $t \in \mathbb{R}$, 
\begin{multline}
|M(t)|^2 = \Big | \int_{\mathbb{R}} e^{i \tau_0 t + i l_{m,n} t} \mathcal{F} [ M] (\tau_0 + l_{m,n} ) 
d\tau_0 \Big |^2 \\
\le \left \{ \int_{\mathbb{R}} \Big | \mathcal{F} [M] (\tau_0 + l_{m,n} ) \Big |^2 
\left ( n^2 + |\tau_0| \right )^{2b}  d\tau_0 \right \}   
\left \{ \int_{\mathbb{R}} \frac{d\tau_0}{(n^2+|\tau_0|)^{2b}} \right \} \\
\lesssim \frac{1}{|n|^{4b-2}} 
\int_{\mathbb{R}} \Big | \mathcal{F} [M] (\tau) \Big |^2 
\left ( n^2 + |\tau-l_{m,n} | \right )^{2b}  d\tau 
\end{multline}
Since the right side is independent of $t$, the lemma follows.
\end{proof}
\begin{Lemma}{(Koenig, Ponce, Vega \cite{KoenigPV}).}
\label{lem4.6}
\begin{equation}
\label{eqlem4.6}
\int_{\mathbb{R}} \frac{d\xi_1}{\left (1 + |\tau-\xi_1^3- (\xi-\xi_1)^3 | \right )^{2b}} 
\lesssim \frac{1}{\sqrt{|\xi|} \left (1 + |4 \tau-\xi^3 | \right )^{1/2}}  
\end{equation}
\end{Lemma}
\begin{proof}
We introduce change of variables
\begin{equation}
\mu = \tau-\xi_1^3 - (\xi-\xi_1)^3 
\end{equation}
Then, it is easily checked that 
$$ d\mu = 3 \xi (\xi-2 \xi_1 ) d\xi_1 \ , \xi_1 = \frac{1}{2} \left [ \xi \pm \sqrt{\frac{4 \tau-\xi^3-4 \mu}{3 \xi}}
\right ] $$
and so 
$$ |\xi(\xi-2 \xi_1)| = 3^{-1/2} \sqrt{|\xi|} \sqrt{|4 \tau - \xi^3 - 4 \mu|}
$$
and therefore
\begin{equation}
\begin{aligned}
\int_{\mathbb{R}} \frac{d\xi_1}{
\left (1 + |\tau-\xi_1^3- (\xi-\xi_1)^3 | \right )^{2b}} 
&\le \frac{1}{\sqrt{3 |\xi|}} 
\int_{\mathbb{R}} \frac{d\mu}{\left ( 4\tau-\xi^3-4\mu \right )^{1/2} (1+|\mu|)^{2b}}\\
&\lesssim \frac{1}{\sqrt{|\xi|} \left ( 1+ |4 \tau- \xi^3 | \right )^{1/2}} 
\end{aligned}
\end{equation}
\end{proof}

\section{Acknowledgments} 
The research of S.T was partially supported by the NSF grant DMS-1108794. Additionally, S.T acknowledges support from the Math Departments, UIC, UChicago and Imperial College during this author's 
sabbatical visits. 
This work was started during C.T's post-doctoral studies at the Ohio State University, whose support she gratefully acknowledges.

\vfill \eject
\end{document}